\documentclass[a4paper,11pt]{article}

\usepackage{latexsym,amsmath,amssymb,amscd,amsfonts,amsthm,mathrsfs}
\usepackage[latin1]{inputenc}
\usepackage{color}
\usepackage{geometry}
\usepackage{dsfont}
\usepackage{bbm}
\usepackage{fullpage}

\theoremstyle{plain}

\newtheorem{thm}{Theorem}
\newtheorem{prop}{Proposition}

\newtheorem{lem}{Lemma}

\def \R {\mathbb R}

\def \E {\mathbb E}

\def \bh {\mathbf{h}}

\newcommand{\ft}[1]{{#1}^*}

\begin{document}

\title{Minimax rates of convergence for Wasserstein deconvolution  with supersmooth errors in any dimension}

\author{J\'er\^ome Dedecker$^{(1)}$ and Bertrand Michel$^{(2)}$}

\maketitle

\noindent{\small
(1)  Laboratoire MAP5  UMR  CNRS 8145,  Universit\'e Paris  Descartes,
Sorbonne Paris Cit\'e,\\
(2) Laboratoire de Statistique Th\'eorique et Appliqu\'ee, Universit\'e Pierre et Marie Curie - Paris 6}

\begin{abstract}
The subject of this paper is the estimation of a probability measure on ${\mathbb R}^d$ from  data  observed with an additive noise, under the
Wasserstein metric of order $p$ (with $p\geq 1$).
We assume that the distribution of the errors is  known and belongs to a class of supersmooth distributions, and  we give optimal rates of convergence for
the Wasserstein metric of order $p$.
In particular, we show how to use  the existing lower bounds for the estimation  of the cumulative distribution function in dimension one to find
lower bounds for the Wasserstein deconvolution  in any dimension.
\end{abstract}

\medskip

{\small

}

\section{Introduction}

We observe $n$ random vectors $  Y_i $ in $\R^d$ sampled according to the convolution model:
\begin{equation} \label{modelConv}
  Y_i = X_i + \varepsilon_i
\end{equation}
where the random vectors $X_i = (X_{i,1},\dots,X_{i,j},\dots,X_{i,d})'$ are i.i.d. and distributed according to an unknown probability measure $\mu$.
The random vectors
$\varepsilon_i = (\varepsilon_{i,1},\dots,\varepsilon_{i,j},\dots,\varepsilon_{i,d})'$ are  i.i.d. and distributed according to a known probability
measure $\mu_{\varepsilon}$. The distribution of the observations $Y_i$ on $\R ^d $ is then the convolution $\mu \star \mu_{\varepsilon}$.
Here, we  shall assume that there exists an invertible matrix $A$ such that the coordinates of the vector $A \varepsilon_1$ are independent
(that is: the image measure of $\mu_\varepsilon$ by $A$ is the product
of its marginals).

This paper is about minimax optimal rates of convergence
for estimating the measure $\mu$ under Wasserstein metrics. For $p \geq 1$, the Wasserstein distance $W_p$ between $\mu$ and $\mu'$ is defined by:
$$W_p(\mu,\mu') = \inf_{\pi \in \Pi(\mu,\mu')} \left ( \int_{\mathbb{R}^d \times \mathbb{R}^d} \| x - y \|^p \pi(dx,dy) \right )^{\frac{1}{p}},$$
where $\Pi(\mu,\mu')$ is the set of probability measures on $\mathbb{R}^d \times \mathbb{R}^d$ with marginals $\mu$ and $\mu'$ and $p $ is
a real number  in $[1, \infty ($ (see \cite{Rachev98b} or \cite{villani2008oton}). The norm $\| .\|$ is the euclidean norm in $\R^d$ corresponding to
the inner product $<\cdot,\cdot>$.

The Wasserstein deconvolution problem is  interesting
in itself since $W_p$ are  natural distances for comparing probability measures. Indeed, contrary to the ${\mathbb L}_p$-distances between
probability densities (except for $p=1$, which coincides with the total variation distance), the distances $W_p$ are true distances between
probability distributions. Note also that many natural estimators $\hat \mu_n$
of 
$\mu$ are singular with respect to $\mu$ (think of the empirical measure
in most cases), and consequently
the total variation distance between $\hat \mu_n$ and $\mu$ is equal to $2$ for any $n$. 
This will be the case of our deconvolution estimator, if the support of $\mu$ is a submanifold in 
${\mathbb R}^d$ with dimension strictly less than $d$. 
Wasserstein metrics  appear as natural distances to evaluate the performance of such estimators.

The Wasserstein deconvolution problem is also
related to  recent results in geometric inference. Indeed, in 2011,  \cite{Chazaletal11} have defined a distance function to a probability
distribution to answer geometric
inference problems in a probabilistic setting. According to their result, the topological properties of a shape can be recovered by using the distance
to a known measure $\tilde \mu$, if $\tilde \mu$ is close enough to a measure $\mu$ concentrated on
this shape
with respect to the Wasserstein distance $W_2$. This fact motivates the study of the Wasserstein deconvolution problem, since in practice the data can
be observed
with noise.

In the paper \cite{CaillerieEtAl2011}, the authors consider a  slight modification of the classical kernel deconvolution estimator, and  they provide
some upper bounds for the rate
 of convergence of this estimator for the $W_2$ distance,  for several noise distributions. Nevertheless the question of optimal rates of convergence in the
minimax sense was left open in this previous work.  The main contribution of the present paper is to find optimal rates of convergence for a class
of supersmooth distributions, for any dimension under any Wasserstein metric $W_p$. In particular  we prove that the deconvolution estimator of $\mu$
under the $W_2$ metric
introduced in~\cite{CaillerieEtAl2011}  is minimax optimal for a class of supersmooth  distributions.

The rates of convergence for  deconvolving  a density  have been deeply studied for other metrics.  Minimax rates in the univariate context can be
found for instance in
\cite{Fan91,ButuceaTsybakov08a,ButuceaTsybakov08b} and  in the recent monograph \cite{Meister09}. The multivariate problem has also been investigated in
\cite{Tang94,ComteLacour2013}. All these contributions concern pointwise convergences or ${\mathbb L}^2$ convergences; rates of convergence for the
Wasserstein metrics have  been studied only by \cite{CaillerieEtAl2011}. In Section \ref{sec:lowerB} of the present paper, we shall see
that, in the supersmooth case, lower bounds for
the Wasserstein deconvolution problem in any dimension can be deduced from lower bounds for the deconvolution of the cumulative distribution function
(c.d.f.) in dimension one.

Another interesting related work is \cite{GenoveseEtAl2012}. In this recent paper, the authors find lower and upper bounds for the risk of estimating
a manifold in Hausdorff distance under several noise assumptions. They consider in particular the additive noise model (\ref{modelConv}) with a
standard multivariate Gaussian noise.

Before giving the main result of our paper, we  need  some notations. Let $\nu$ be a measure on ${\mathbb R}^d$ with density $g$ and let $m$ be
another measure on $\R^d$. In the following we shall denote by $m \star
g $ the density of $m \star \nu$, that is
$$
  m \star g (x)= \int_{{\mathbb R}^d} g(x- z) m(dz) \, .
$$
We also denote by $\ft{\mu}$ (respectively  $\ft{f}$) the Fourier transform of the probability measure $\mu$ (respectively of the integrable function
$f$), that is:
$$
    \ft{\mu}(x)= \int_{{\mathbb R}^d} e^{i<t,x>} \mu(dt) \quad \text{and} \quad
    \ft{f}(x)= \int_{{\mathbb R}^d} e^{i<t,x>}  f(t) dt \, .
$$
For a $d\times d$ matrix $A$ and some constants $M > 0$, $p \geq 1$ and $a >1$,  let $\mathcal D_A(M,p,a)$ be the set of measures $\mu$ on
$\R^d$ for which
\begin{equation} \label{eq:DAM}
\sup_{1 \leq j \leq d}\E_{\mu} \Big( (    1 +
 | (A X_1)_j |^{2p +a} )\prod_{1 \leq \ell \leq d,  \, \ell \neq j}   (    1 +
  |(A X_1)_\ell|^{a}  )\Big) \leq M < \infty \, .
\end{equation}
Moreover we simply use the notation  $\mathcal D(M,p, a)$ if $A = \operatorname I _d$. Note that Condition (\ref{eq:DAM}) requires at least moment of
order $2p+a$ on each coordinate. In the case where the $(A X_1)_j$'s are
independent, this condition is satisfied when the $(A X_1)_j$'s have a moment of order $2p+a$. If, for some $k_0 \in \{1,\dots d\}$ all the $(A
X_1)_j$ for $j \neq k_0$ are bounded, then one need only a moment of order $2p+a$ for $(A X_1)_{k_0}$.

Let us give the main result of our paper when $\varepsilon_1$ is a  non degenerate Gaussian random vector (by non degenerate, we mean that
its covariance matrix is not equal to zero).

\begin{thm}\label{ThG}
Assume that we observe $Y_1, \dots, Y_n$ in the multivariate convolution model
(\ref{modelConv}), where $\varepsilon_1$ is a non degenerate Gaussian random vector.
Let  $A$ be an invertible matrix such that the coordinates of $A \varepsilon_1$
are independent.
Let $M>0$, $p\geq 1$  and $a>1$. Then
\begin{enumerate}
\item
There exists a constant $C>0$  such that for any estimator $\tilde \mu_n$ of the
measure $\mu$:
$$
\liminf _{ n \rightarrow \infty} \  (\log n) ^{p / 2}   \sup_{\mu \in \mathcal D_A(M,p,a)}\,  \E_{ (\mu \star \mu_\varepsilon) ^{\otimes n }} ( W_p^p
(\tilde \mu_n, \mu) ) \geq
C .
$$
\item One can build an estimator $\hat \mu_n$  of $\mu$ such that:
$$
\sup_{n \geq 1} \sup_{\mu \in \mathcal D_A(M,p,a)} \ (\log n) ^{p / 2} \
\E_{ (\mu \star \mu_\varepsilon) ^{\otimes n }}  ( W_p^p (\hat \mu_n, \mu) ) \leq
K \, ,
$$
for some positive constant $K$.
\end{enumerate}
\end{thm}
Note that in Theorem \ref{ThG} the random vector $\varepsilon_1$
may have all its coordinates, excepts one,  equal to zero almost surely.
In other words, a Gaussian noise in one direction  leads to the
same rate of convergence as an isotropic Gaussian noise.

The paper is organized as follows. The proof of the lower bound is given in Section~\ref{sec:lowerB}. In Section~\ref{sec:upperB} we then give the
corresponding upper bound in the same context by generalizing the results of \cite{CaillerieEtAl2011} for all $p \geq 1$. We finally discuss the $W_p$
deconvolution problem for ordinary smooth case in Section~\ref{sec:Disc}. Some additional technical results are given in Appendix.

\section{Lower bounds} \label{sec:lowerB}
\subsection{Main result}

The following theorem is the main result of this section. It gives a lower bound on the rates of convergence of measure estimators in the supersmooth
case for any dimension and under any
metric $W_p$. 

\begin{thm} \label{theo:LBWp}
Let $M > 0$, $p \geq 1$ and $a>1$. Assume that we observe $Y_1, \dots, Y_n$ in the multivariate convolution model
(\ref{modelConv}). Assume that there exists $j_0 \in \{1,\dots,d\}$ such that
the coordinate  $(A \varepsilon_1)_{j_0}$ has a density $g$ with respect to
the Lebesgue measure satisfying for all $w \in \R$:
 \begin{equation} \label{Assum:supsm}
| \ft{g}(w)| (1+|w|)^ {-\tilde \beta} \exp(|w|^{\beta} / \gamma_1 )  \leq c_1
\end{equation}
for some $\beta > 0$ and some $\tilde \beta \in {\mathbb R}$. Also assume that there exist some constants $\kappa_1 \in (0,1)$ and  $\kappa_2 > 1$  such that
\begin{equation} \label{Assum:supsm+}
 P(|(A \varepsilon_1)_{j_0} - t | \leq |t|^{\kappa_1}) = O(|t|^{-\kappa_2}) \quad \textrm{as } |t| \rightarrow \infty
\end{equation}
and
 \begin{equation}  \label{eq:pkappa12}
\max \left( p + 1+ \frac a 2  \, , \,  \frac {\kappa_2}   {2 \kappa_1} + \frac 1 2  \right) < \kappa_2  .
\end{equation}
Then  there exists a constant $C>0$ such that for all estimator $\tilde
\mu_n$ of the
measure $\mu$:
$$
\liminf_{n \rightarrow \infty} \ (\log n) ^{p / \beta}   \sup_{\mu \in \mathcal D_A(M,p,a)}\,  \E_{ (\mu \star \mu_\varepsilon) ^{\otimes n }} \, W_p^p
(\tilde \mu_n, \mu) \geq
C .
$$
\end{thm}

The assumption about the random variable $(A\varepsilon_1)_{j_0}$  means that the noise is supersmooth in at
least one direction. Indeed, as shown in Section  \ref{cdf}, the lower bound for the multivariate problem can be deduced from
the lower bound for the ${\mathbb L}^1$ estimation of the c.d.f.
of $(A\varepsilon_1)_{j_0}$. If the distribution of the noise is supersmooth in several directions
then one may choose the direction with the greatest coefficient $\beta$.

The assumption (\ref{Assum:supsm+}) is classical in the deconvolution setting, see for instance \cite{Fan91,Fan92}. The technical assumption
(\ref{eq:pkappa12}) summarizes the conditions on $p$ and $\kappa_2$. The condition $\frac {\kappa_2}   { 2 \kappa_1} +
\frac 1 2  < \kappa_2 $ is also required in \cite{Fan91} and \cite{Fan92}. The additional
condition $p  + 1 + \frac a 2  < \kappa_2$ is a consequence of the moment assumption on $\mu$.

If the noise distribution has finite moment of order $p+b$ for some $b> 1 + \frac a 2 $, we can state the next lemma. This moment condition is always
satisfied under the assumptions used to prove  the upper bound  
(see Theorem \ref{upperbound}).
\begin{lem} \label{lem:Cdtions}
Assume that $ \E |(A \varepsilon_1)_{j_0} |^{p+b}  < \infty $ for some $b>1 + \frac a 2$. Then
one can find $\kappa_1 \in (0,1)$ and $\kappa_2 > 1$ such that Conditions (\ref{Assum:supsm+}) and (\ref{eq:pkappa12}) are satisfied.
\end{lem}
\begin{proof}

%
Since $ \E |(A \varepsilon_1)_{j_0} |^{p+b}  < \infty $  then $ P(  |(A \varepsilon_1)_{j_0} | \geq  |t| ) =  O (|t|^{-p-b})$ as $|t|$ tends to infinity. We take $\kappa_2 = p+b$ and we get $p + 1 +\frac a 2 < \kappa_2$.  For any $\kappa_1 \in (0,1)$,  for $|t|$ large enough:
\begin{align*}
 P\left(|(A \varepsilon_1)_{j_0} - t | \leq |t|^{\kappa_1} \right)   
 \leq  P\left(  | (A \varepsilon_1)_{j_0} |  \geq   | t  | - |t|^ {\kappa_1}   \right) \\
 \leq  P \left( | (A \varepsilon_1)_{j_0} | \geq \frac {|t|} 2 \right) =  O (|t|^{-\kappa_2}).
 \end{align*}
Thus (\ref{Assum:supsm+})  is satisfied for any $\kappa_1 \in (0,1)$. Finally, we choose $\kappa_1$ close enough to 1 to satisfy (\ref{eq:pkappa12}).
\end{proof}

\subsubsection{Wasserstein deconvolution and c.d.f. deconvolution}
\label{cdf}

It is well known that the Wasserstein distance $W_1$
between two measures $\mu$ and $\mu'$ on $\R$ can be computed using the cumulative distribution functions:
  Let $\mu$ and $\mu'$ be two probability measures on $\R$, then
\begin{equation*}\label{connu}
 W_1(\mu,\mu') =  \int_{\R}| F_{\mu }(x)  -  F_{\mu'}(x) | \, dx 
\end{equation*}
where $F_{\mu }$ denote the c.d.f. of $\mu$. According to this property,
lower bounds on the rates of convergence for estimating $\mu$ in the
one dimensional convolution model (\ref{modelConv})  for the
metric $W_1$ can be directly deduced from lower bounds on the rates of convergence
for the estimation of the c.d.f. of $\mu$ using the integrated
risk $\mathcal R (\hat F ) := \int_\R |F_{\mu}(t) - \hat F (t) | dt $. This last problem has been less studied than pointwise rates in the
deconvolution context  but some results can be found in the literature. For instance \cite{Fan92} gives the optimal rate of convergence in
the supersmooth case for an integrated (weighted) $L_p$ risk under similar
smoothness conditions as for the pointwise case (studied in \cite{Fan91}).
The {\it cubical method}
followed in \cite{Fan92} to compute the integrated lower bound is also detailed in \cite{Fan93}. It is based on   a multiple hypothesis
strategy, see \cite{Tsybakov09} for other examples of using multiple hypothesis schema for computing  lower bounds for integrated risks.

For $M > 0$,  $p \geq 1$ and $a>1$,  we consider the set $\mathcal C_A (M,p, a)$ of the measures $\mu$ in $\mathcal D_A(M,p,a)$ for which the coordinates of
$A X_1$ are independent. Thus, for  $\mu \in \mathcal C_A (M,p, a)$, 
\begin{equation*}
\sup_{1\leq j \leq d}\Big (\E_{\mu} (    1 +  | (A X_1)_j |^{2p+a}    )\prod_{1\leq \ell \leq d
, \, \ell \neq j} \E_{\mu} (   1 +  |(A X_1)_\ell|^{a}   )\Big) \leq M < \infty .
\end{equation*}
Moreover we simply use the notation  $\mathcal C(M,p, a)$ if $A = \operatorname I _d$.

The following theorem  gives  lower bounds for $W_1(\tilde \mu_n, \mu)$
in the $d$-dimensional case, which are derived from  lower bounds on the rates of convergence of c.d.f. estimators in $\R$.
\begin{thm} \label{theo:lbW1}
Under the same assumptions as in Theorem~\ref{theo:LBWp}, there exists $C>0$  such that for all estimator $\tilde
\mu_n$ of the
measure $\mu$:
$$
\liminf_{n \rightarrow \infty} \ (\log n) ^{1 / \beta}  \sup_{\mu \in \mathcal C_A(M,p,a)}\,  \E_{ (\mu \star \mu_\varepsilon) ^{\otimes n }} \, W_1
(\tilde \mu_n, \mu) \geq C
.
$$
\end{thm}

Theorem~\ref{theo:LBWp}  is a corollary of Theorem~\ref{theo:lbW1} because
\begin{enumerate}
\item $\mathcal C_A(M,p,a)$ is a subset of $\mathcal D_A (M,p,a)$.
\item For any $p\geq 1$, $ \E_{ (\mu \star \mu_\varepsilon) ^{\otimes n }} \, W_1^p
(\tilde \mu_n, \mu) \geq \left(\E_{ (\mu \star \mu_\varepsilon) ^{\otimes n }} \, W_1
(\tilde \mu_n, \mu)\right)^p$.
\item $W_1$ is the smallest among all the
Wasserstein distances: for any $p \geq 1 $ and any measures $\mu$ and $\mu'$ on $\R^d$:
$  W_p(\mu,\mu')  \geq  W_1(\mu,\mu')$.
\end{enumerate}

\subsection{Proof of Theorem \ref{theo:lbW1}}

Since the works of Le Cam, it is well known that rates of convergence of estimators on some probably measure space $\mathcal P$ can be lower bounded
by introducing some convenient  finite subset of $\mathcal P$  whose elements are close enough for the total variation distance or for the Hellinger
distance. In the deconvolution setting, $\chi^2$ distance are preferable to these last metrics. Here, the following definition of the
$\chi^2$ distance will be
sufficient: for
two positive densities  $h_1$ and $h_2$  with respect to  the Lebesgue measure on $\R^d$,
 the $\chi ^2$
distance
between $h_1$ and $h_2$ is defined by
$$ \chi^2(h_1,h_2)  = \int_{\R^d} \frac{\left\{(h_1(x) -  h_2 (x)\right\}^2 }{h_1(x)} d x.$$

The main arguments for proving Theorem~\ref{theo:lbW1} come from \cite{Fan91,Fan92,Fan93}. However
some modifications are necessary to compute the lower bounds under the moment assumption $\mathcal C_A (M,p,a)$. Furthermore, we note that Theorem~1
in \cite{Fan93} cannot be directly applied in this multivariate context.

Without loss of generality, we take $j_0 = 1$.
We shall first prove Theorem~\ref{theo:lbW1} in the case where $\varepsilon_1$
has independent coordinates.

\subsubsection{Errors with independent coordinates}\label{eic}

In this section, we observe $Y_1, \dots, Y_n$ in the multivariate convolution model (\ref{modelConv}) and we assume that the random variables
$(\varepsilon_{1,j})_{1\leq j \leq d}$ are independent. This means that $A = \operatorname{I}_d$ and
that $\varepsilon_1$ has the distribution $\mu_\varepsilon= \mu_{\varepsilon,1} \otimes \mu_{\varepsilon,2} \otimes \dots \otimes
\mu_{\varepsilon,d}$.

\medskip

\noindent {\bf Definition of a finite family in $\mathcal C (M,p,a)$}.
 Let us introduce a finite class of probability measures in $\mathcal C (M,p,a)$ which are absolutely continuous with respect
to the Lebesgue measure $\lambda_d$. First, we define some densities
\begin{equation} \label{eq:f0}
 f_{0,r}(t) := C_r (1+ t^2)^{-r}
\end{equation}
with some $r >  0 $ such that
 \begin{equation}  \label{eq:rpkappa12}
\max \left( p + \frac 1 2 +\frac a 2 \, , \, \frac 1 2 \frac {\kappa_2}   {\kappa_1}   \right) <  r < \kappa_2 - \frac 1 2 .
\end{equation}
Note that this is possible according to (\ref{eq:pkappa12}). For such a $r$,  $f_{0,r}$ has a finite $(2p+a)$-th moment.

Next, let $ b_n$ be the sequence
\begin{equation} \label{eq:defbn}
b_n := \Big[\Big(\frac 1 \eta \log n \Big)^{1/ \beta}\Big ]\vee 1 \, ,
\end{equation}
where $[\cdot]$ is the integer part,
and $\eta =  \left(1 - \frac{2r}{2\kappa_2 -1} \right) / \gamma$. Note that $b_n$ is correctly defined in this way since  $
\kappa_2 - \frac 1 2 > r$. For any $\theta \in
\{0,1\} ^{b_n}$, let
\begin{equation} \label{eq:ftheta}
f_{\theta} (t) = f_{0,r}(t) + C \sum _{s=1} ^{b_n} \theta_s  H \left(b_n ( t - t_{s,n}) \right), \quad t \in \R,
\end{equation}
where $C$ is a positive constant and $t_{s,n} = (s-1) / b_n$. The function $H$ is a bounded function whose integral on the line is $0$.
Moreover, we
may choose a function $H$ such that (see for instance \cite{Fan91} or \cite{Fan93}):
\begin{description}
\item (A1) $\int_{-\infty}^{+\infty} H (t)\,  dt  = 0$ and $ \int_ 0 ^1  |H^{(-1)}(t)| \, dt >0 $,
\item (A2) $|H(t)| \leq c (1 + t^2) ^{-r}$,
\item (A3) $\ft{H}(z) = 0 $ outside $[1,2]$
\end{description}
where $H^{(-1)}(t) : = \int_{-\infty} ^t H(u)  \, d u $ is a primitive of $H$.

Using (A2) and Lemma~\ref{Lemma1FanTruong93} of Appendix \ref{A}, we choose $C >0$ small enough in such a way that
$f_{\theta}$ is a density on
$\R$. Note that by replacing $H$ by $H/C$ in the following, we finally can take $C = 1$ in (\ref{eq:ftheta}). Using (A2), Condition (\ref{eq:rpkappa12}) and
Lemma~\ref{Lemma1FanTruong93} again, we can find 
some $M$ large enough such that
for all $\theta \in
\{0,1\} ^{b_n}$:
\begin{equation} \label{FninCmbB6}
 \int_\R  \left( 1 + |t|^a  \vee |t| ^{2p+a} \right) f_{\theta}(t) \, d t \leq  M^{1/ d } .
\end{equation}
We finally use these univariate densities $f_\theta$ to define a finite family of probability measures on $\R^d$
which is included in $\mathcal C(M,p,a)$. For $\theta \in \{0,1\}
^{b_n}$, let us define the probability measure on $\R^d$:
\begin{equation} \label{eq:mutheta}
\mu_{\theta} := \left(f_{\theta} \cdot d \lambda \right) \otimes \left( f_{0,r} \cdot d \lambda \right) \otimes \dots \otimes \left( f_{0,r} \cdot d
\lambda \right)  .
\end{equation}
For any $j \in \{1,\dots,d \}$, according to (\ref{FninCmbB6}) :
\begin{equation*}
\Big(\E_{\mu_{\theta} } (    1 +    |X_{1,j}|^{2p+a}    )\prod_{2\leq \ell \leq d} \E_{\mu_{\theta} } (   1 +  |X_{1,\ell}|^{a}   ) \Big)
  \leq  M
\end{equation*}
and thus $\mu_{\theta} \in \mathcal C(M,p,a)$.

\medskip

\noindent{\bf Lower bound.}
Let $\tilde \mu _n $ be an estimator of $\mu$ and let $(\tilde \mu _n)_1$ be the marginal distribution of $\tilde \mu _n $ on the first coordinate
(conditionally to the sample $Y_1, \dots, Y_n$). According to Lemma~\ref{MinorW1} of Appendix \ref{B}:
\begin{align*}
\sup _{\mu \in \mathcal C(M,p,a)} \E _{  (\mu \star \mu_{\varepsilon}) ^{\otimes n }}W_1 \left(\mu , \tilde \mu _n  \right)
& \geq  \sup _{\theta \in \{ 0,1\}^n }\E _{  (\mu_\theta \star \mu_{\varepsilon}) ^{\otimes n }} W_1 \left(\mu_\theta , \, \tilde
\mu _n  \right) \\
& \geq  \sup _{\theta \in \{ 0,1\}^n } \E _{  (\mu_\theta \star \mu_{\varepsilon}) ^{\otimes n }} W_1 \left( f_\theta  \cdot
d\lambda \, , \, (\tilde \mu _n )_1 \right) \\
& \geq  \inf_{\hat f _n}  \sup _{\theta \in \{ 0,1\}^n } \E _{  (\mu_\theta \star \mu_{\varepsilon}) ^{\otimes n }} W_1 \left( f_\theta  \cdot
d\lambda \,
, \, \hat f _n \right)
\end{align*}
where the infimum of the last line is taken over all the probability measure estimators of $ f_\theta  \cdot
d\lambda$.

Following \cite{Fan93} (see also the proof of Theorem 2.14 in \cite{Meister09}), we now introduce  a random vector $\tilde \theta$ whose components $\tilde
\theta_s$ are i.i.d. Bernoulli random variables $\tilde \theta_1,\ldots,\tilde \theta_{b_n}$ such that $P(\tilde \theta_s = 1) = \frac 1 2$. The
density $f_{\tilde \theta}$ is thus a random density taking its values in the set of densities defined by (\ref{eq:ftheta}). Let $\E$ be the
expectation according to the law of ${\tilde \theta}$. For any
probability estimator $\hat f_n  $:
\begin{align}
\sup _{\mu \in \mathcal C(M,p,a)} \E _{  (\mu \star \mu_{\varepsilon}) ^{\otimes n }} W_1 \left(\mu , \tilde \mu _n  \right)
& \geq  \E   \: \E _{  (\mu_{\tilde \theta} \star \mu_{\varepsilon}) ^{\otimes n }} \left[ W_1 \left(f_{\tilde \theta}  \cdot d\lambda \, , \, \hat
f_n
 \right)  \right] \notag \\
& \geq   \int_\R  \E   \: \E _{  (\mu_{\tilde \theta} \star \mu_{\varepsilon}) ^{\otimes n }} \left( |F_{\tilde \theta}(t)  - \hat F_n(t)  |  \right)
\, dt \label{W1F}
\end{align}
 where $\hat F$ and $F_\theta$ are the c.d.f. of the distributions $\hat f _n $ and $f _\theta \cdot d \lambda$. For $\theta \in
\{0,1\}^{b_n}$ and $s \in \{1,\dots, b_n\}$, let us define
$$f_{\theta,s,0} := f_{(\theta_1, \dots,\theta_{s-1},0 ,\theta_{s+1}, \dots,\theta_{b_n})} \quad \textrm{ and } f_{\theta,s,1} := f_{(\theta_1,
\dots,\theta_{s-1},1,\theta_{s+1}, \dots,\theta_{b_n})} $$
and the corresponding probability measures $\mu_{\theta,s,0}$ and $\mu_{\theta,s,1}$ on $\R^d$ defined by (\ref{eq:mutheta}) for $f_{\theta} =
f_{\theta,s,0}$ or $ f_{\theta,s,1}$. Let $\bar h_{\theta,s,0} $ and $\bar h_{\theta,s,1} $ be the densities of  $\mu_{\theta,s,0} \star
\mu_{\varepsilon}$ and $ \mu_{\theta,s,1} \star \mu_{\varepsilon}$  for the Lebesgue measure on $\R^d$. Since the margins of $ \mu_\theta$  and
$\mu_\varepsilon$ are independent, for any $y_i = ( y_{i,1}, \dots, y _{i,j} , \dots , y_{i,d})  \in \R^d$ ($u=0$ or $1$), we have:
\begin{equation} \label{eq:hbardec}
\bar h_{\theta,s,u}(y_i) =  h_{\theta,s,u}(y_{i,1})  \prod _{j=2, \dots , d }   f_{0,r} \star \mu_{\varepsilon,j} (y_{i,j})
\end{equation}
where $h_{\theta,s,u} = f _{\theta,s,u} \star g$.
Let $F_{\theta,s,0}$ and $F_{\theta,s,1}$ be the c.d.f of $f_{\theta,s,0} $ and $f_{\theta,s,1}$. For $t \in [t_{s,n},t_{s+1,n}]$ where $s$ in
$\{1,\dots, b_n\}$, by conditioning by $\tilde \theta_s$, we find that
\begin{multline*}  \label{Ftomin}
 \E   \: \E _{  (\mu_{\tilde \theta} \star \mu_{\varepsilon}) ^{\otimes n }} \left( |F_{\tilde \theta}(t)  - \hat F_n(t)  |  \right)
= \\ \frac 1 2 \E   \left[ \: \E _{ \bar  h_{\tilde \theta,s,0} ^{\otimes n } } \left( |F_{\tilde \theta,s,0}(t)  - \hat F_n(t)  |  \right) \: +   \:
\E _{ \bar h_{\tilde \theta,s,1}  ^{\otimes n } } \left( |F_{\tilde \theta,s,1}(t)  - \hat F_n(t)  |  \right) \right] \, .
\end{multline*}
Hence
\begin{multline*}
\E   \: \E _{  (\mu_{\tilde \theta} \star \mu_{\varepsilon}) ^{\otimes n }} \left( |F_{\tilde \theta}(t)  - \hat F_n(t)  |  \right)
 \geq \frac 1 2 \E \int_{\R^d} \dots  \int_{\R^d}  \left\{  |F_{\tilde \theta,s,0}(t)  - \hat F_n(t)  | + |F_{\tilde \theta,s,1}(t)  - \hat F_n(t)   | \right\} \\
 \min \left( \prod_{i=1}^n  \bar h_{\tilde \theta,s,0}(y_i),   \prod_{i=1}^n \bar h_{\tilde \theta,s,1}(y_i)  \right)  d y_1 \dots d y_n \, ,
 \end{multline*}
 and consequently, according to  (\ref{eq:hbardec}),
 \begin{multline*}
\E   \: \E _{  (\mu_{\tilde \theta} \star \mu_{\varepsilon}) ^{\otimes n }} \left( |F_{\tilde \theta}(t)  - \hat F_n(t)  |  \right)\geq \frac 1 2 \E \int_{\R^d} \dots  \int_{\R^d} |F_{\tilde \theta,s,0}(t)  - F_{\tilde \theta,s,1}(t)   | \\ \min \left( \prod_{i=1}^n h_{\tilde
\theta,s,0}(y_{i,1}),  \prod_{i=1}^n  h_{\tilde \theta,s,1}(y_{i,1})  \right)   \left\{ \prod _{j=2} ^d f_{0,r} \star \mu_{\varepsilon,j}
(y_{i,j}) \right\} d y_1 \dots d y_n \, .
\end{multline*}
By using Fubini, it follows that
\begin{multline*}
\E   \: \E _{  (\mu_{\tilde \theta} \star \mu_{\varepsilon}) ^{\otimes n }} \left( |F_{\tilde \theta}(t)  - \hat F_n(t)  |  \right)
\geq \\\frac 1 2 \E\int_{\R} \dots  \int_{\R}  |F_{\tilde \theta,s,0}(t)  - F_{\tilde \theta,s,1}(t)   |  \min \left( \prod_{i=1}^n h_{\tilde
\theta,s,0}(y_{i,1}),  \prod_{i=1}^n  h_{\tilde \theta,s,1}(y_{i,1})  \right)   dy_{1,1} \dots d y_{n,1}\, .
\end{multline*}
Note that for any $\theta  \in \{0,1\} ^{b_n}$, $|F_{\tilde \theta,s,0}(t)  - F_{\tilde \theta,s,1}(t)   | = b_n^{-1}  \left| H ^{(-1)}  \left (b_n (t - t_{s,n})\right) \right| $, thus
\begin{multline}
\E \: \E_{  (\mu_{\tilde \theta} \star \mu_{\varepsilon}) ^{\otimes n }}
\left( |F_{\tilde \theta}(t)  - \hat F_n(t)  | \right) \geq \\
\ \frac {\left| H ^{(-1)} \left(b_n (t - t_{s,n})\right) \right| } {2 b_n }   \E   \, \int_{\R^n}   \min \left( \prod_{i=1}^n h_{\tilde
\theta,s,0}(y_{i,1}),  \prod_{i=1}^n  h_{\tilde \theta,s,1}(y_{i,1})  \right)   d y_{1,1} \dots d y_{n,1}.
\end{multline}
According to Le Cam's Lemma (see Lemma~\ref{lem:LeCam} of  Appendix \ref{B}), for any $\theta \in \{0,1\}^{b_n}$:
\begin{multline} \label{minchi2}
\int_{\R^n} \min \left( \prod_{i=1}^n  h_{\theta,s,0}(y_{i,1}),  \prod_{i=1}^n h_{ \theta,s,1}(y_{i,1})   \right) d y_{1,1} \dots d y_{n,1}   \\
\geq \frac 1 2   \left[\int_{\R^n}    \left\{  \prod_{i=1}^n  h_{\theta,s,0}(y_{i,1}) \right\}^{\frac 1 2 }\left\{  \prod_{i=1}^n h_{
\theta,s,1}(y_{i,1}) \right\} ^{\frac 1 2 }dy_{1,1} \dots d y_{n,1}\right]^2 \\
\geq \frac 1 2   \left[\int_{\R}    \sqrt{ h_{\theta,s,0}(y_{1,1}) \, h_{ \theta,s,1}(y_{1,1}) } d y_{1,1}\right]^{2 n }  \\
 \geq \frac 1 2   \left[ 1 - \frac 1 2 \chi ^2 \left(  h_{\theta,s,0} \, , \,  h_{ \theta,s,1}\right) \right]^{2 n }
\end{multline}
where we have used Lemma~\ref{lem:lbchi2} of Appendix \ref{B} for the last inequality.  Assume for the moment that there exists a constant $c >0$ such that for any $\theta \in \{0,1\}^{b_n}$:
\begin{equation}  \label{eq:chi2h}
\chi ^2 \left(  h_{\theta,s,0}  \, , \, h_{ \theta,s,1} \right) \leq \frac c n.
\end{equation}
Then, using (\ref{W1F}), (\ref{Ftomin}), (\ref{minchi2}) and (\ref{eq:chi2h}), we find that there exists a constant $C >0$ such that
$$
\sup _{\mu \in \mathcal C(M,p, a)} \E _{  (\mu \star \mu_{\varepsilon}) ^{\otimes n }}W_1 \left(\mu , \tilde \mu _n  \right)
 \geq  \frac C { b_n }  \sum_{s = 1}^{b_n}  \int_{t_{s,n}} ^{t_{s+1,n}} \left| H ^{(-1)}  \left(b_n (t - t_{s,n})\right) \right|   \, dt
   \geq  \frac C {b_n }  \int_{0} ^{1} \left|  H ^{(-1)}  (u)  \right| \, d  u \,  .
$$
Take $b_n$ as  in (\ref{eq:defbn}) and the theorem is thus proved (for $A =\operatorname{I}_d$) since the last term  is positive according to (A1).

\vskip 0.5cm

\noindent {\bf Proof of (\ref{eq:chi2h}).}
Let $C$ be a positive constant which may vary from line to line.
 We follow \cite{Fan92} to show that (\ref{eq:chi2h}) is valid for $b_n$ chosen as in (\ref{eq:defbn}).
 Recall that we have chosen the function $H$ such that,
by Lemma~\ref{Lemma1FanTruong93} of Appendix \ref{A},
$f_{\theta} \geq C f_{0,r}$.
 Thus,
 \begin{align*}
\chi ^2 \left(  h_{\theta,s,0} \, , \, h_{ \theta,s,1} \right)
&\leq \int_{-\infty} ^{+\infty}  \frac{ \left\{\int_{-\infty} ^{+\infty}   H
\left[b_n(t-u-t_{s,n})\right] g(u)\, d u \right\}^2} {f_{\theta,s,0} \star g (t)} d t \notag \\
&\leq C  \int_{-\infty} ^{+\infty}  \frac{ \left\{ \int_{-\infty} ^{+\infty}  H
\left[b_n(t-u-t_{s,n})\right] g(u)\, d u \right\}^2} {f_{0,r} \star g (t) } d t \notag  \\
&\leq C  \int_{-\infty} ^{+\infty}  \frac{ \left\{ \int_{-\infty} ^{+\infty}  H
\left[b_n(t'-u)\right] g(u)\, d u \right\}^2} {\int_{-\infty} ^{+\infty} f_{0,r}(t' + t_{s,n}-u)  g (u) \, du} d t'\notag .
\end{align*}
Moreover, there exists a positive constant  $C $ such that for any $t \in \R$ and any $s \in
\{1,\dots,b_n\}$, $ f_{0,r}(t + t_{s,n}) \geq C f_{0,r}(t)$.  Then,
 \begin{align}
\chi ^2 \left(  h_{\theta,s,0} \, , \, h_{ \theta,s,1} \right)
&\leq C     \int_{-\infty} ^{+\infty}  \frac{ \left\{ \int_{-\infty} ^{+\infty}  H
\left[b_n(t'-u)\right] g(u)\, d u \right\}^2} {\int_{-\infty} ^{+\infty} f_{0,r}(t'  -u)   g (u) \, du } d t'\notag  \\
&\leq C  b_n^{-1}   \int_{-\infty} ^{+\infty}  \frac{ \left\{ \int_{-\infty} ^{+\infty}  H
(v-y) g(y/b_n)\, d y  / b_n\right\}^2} {  f_{0,r} \star g (v/b_n)} d v . \label{reducekhi2}
\end{align}
The right side of (\ref{reducekhi2}) is typically the kind of  $\chi ^2$ divergence that is upper bounded in the proof of Theorem~4 in
\cite{Fan91} for computing pointwise rates of convergence. However, a slight modification of the  proof of Fan is necessary since we cannot assume
here that $r < \min(1,\kappa_2-0.5)$ (because $r > p + (1+a)/2$). It is shown in the proof of Theorem~4 in \cite{Fan91} that
\begin{equation} \label{eq:intH}
 \int_{-\infty} ^{+\infty}    \left\{ \int_{-\infty} ^{+\infty}  H (v-y) g(y/b_n)\, d y  / b_n \right\}^2  \, d v = O \left( b_n ^{2 \tilde \beta}
\exp( - 2 b_n ^\beta  / \gamma) \right) .
\end{equation}
According to Lemma~\ref{Lemma51Fan91} of Appendix \ref{A}, there exist $t_0 >0$, $C_{1} >0$ and $C_{2} >0$ such that for any $t \in \R$:
\begin{equation} \label{eq:consLemma51}
f_{0,r} \star g (t) \geq C_{1} \mathds{1}_{|t| \leq t_0} +   \frac{C_{2}}{ t ^{2r} } \mathds{1}_{|t| > t_0}
\end{equation}
Note that we can apply Lemma~\ref{Lemma52Fan91} of appendix \ref{A} since $r$ satisfies (\ref{eq:rpkappa12}). Then, using (\ref{eq:intH}),
(\ref{eq:consLemma51}) and Lemma~\ref{Lemma52Fan91} of Appendix \ref{A}, for $T > t_0$  we have:
\begin{align*}
  & \int_{-\infty} ^{+\infty}  \frac{ \left\{ \int_{-\infty} ^{+\infty}  H (v-y) g(y/b_n)\, d y  / b_n\right\}^2} {  f_{0,r} \star g (v/b_n)} d v   \\
 &  =  \int_{|v |/ b_n   \leq T}  \frac{ \left\{ \int_{-\infty} ^{+\infty}  H (v-y) g(y/b_n)\, d y  / b_n\right\}^2} {  f_{0,r} \star g (v/b_n)} d v
    + \int_{|v |/ b_n > T}  \frac{ \left\{ \int_{-\infty} ^{+\infty}  H (v-y) g(y/b_n)\, d y  / b_n\right\}^2} {  f_{0,r} \star g (v/b_n)} d v \\
& \leq (C_{1} \wedge C_{2} T ^{-2 r}) ^{-1} \int_{-\infty} ^{+\infty}    \left\{ \int_{-\infty} ^{+\infty}  H (v-y) g(y/b_n)\, d y  / b_n
\right\}^2 \, d v
+ C_r  \int_{|v |/ b_n > T}  \frac{    (|v|/b_n) ^{-2 \kappa_2} }  {   (|v|/b_n)^{-2r} } d v \\
& \leq       O \left(T ^{2 r}  b_n ^{2 \tilde \beta} \exp( -
2 b_n ^\beta  / \gamma) \right)
+   O \left(  b_n^{  2 (r -\kappa_2)}  T  ^{  2 (r -\kappa_2) +1 } \right)
\end{align*}
for $T$ large enough. By taking $T = T_n = b_n^{\frac{2  r -2 \kappa_2 - 2 \tilde \beta}{2 \kappa_2-1} } \exp \left(\frac {2 b_n^\beta }{\gamma (
2\kappa_2-1)}\right)$ in this bound and
according to (\ref{reducekhi2}), we find that for $n$ large enough:
\begin{align*}
\chi^2  (  h_{\theta,s,0} \, , \, h_{ \theta,s,1}  )
& =  O \left(  b_n^{ 2 \tilde \beta + 4 r \frac{r -\kappa_2 -\tilde \beta}{2 \kappa_2-1}   } \exp \left\{- \frac {2 b_n^\beta
}{\gamma}
\left[1 -\frac{2r}{2\kappa_2 -1}  \right] \right\} \right) \label{Okhi2} \notag \\
& =  O  \left( \exp(- \eta b_n ^{\beta})\right) = O \left( \frac 1 n \right)
\end{align*}
for $b_n$ defined by (\ref{eq:defbn}).

\subsubsection{The general case}

We now assume, as in the introduction, that there exists an invertible matrix $A$ such that the coordinates of the vector $A \varepsilon_1$ are
independent. Let $\mu \in \mathcal C_A(M,p,a)$ and let $\hat \mu _n$ be an estimator of the probability measure $\mu$. Let $\mu^{A}$  and $\hat \mu_n
^{A}$ be the image measures of $\mu$ and $\hat \mu$ by $A$. Then,
\begin{align*}
  W_1(\hat \mu ^{A}_n , \mu^{A} ) &=
  \min_{\pi \in \Pi(  \hat \mu _n ^{A} , \mu^{A} )} \int_{\R^d \times \R^d }  \|x-y\| \pi(dx,dy) \\
  &=\min_{ \tau \in \Pi(\hat \mu_n, \mu)} \int_{\R^d \times \R^d } \|Ax-Ay\| \, \tau(dx,dy) \\
  & \leq \|A\| \, W_1 \left( \hat \mu_n  , \mu   \right) \, ,
\end{align*}
where $\|A\|=\sup_{\|x\|=1} \|Ax\|$.
Consequently
$
W_1 \left( \hat \mu_n  , \mu   \right) \geq \|A\|^{-1} W_1(\hat \mu ^{A}_n , \mu^{A} )
$.

The image measure of $\mu \star \mu_\varepsilon$ by $A$ is equal to $\mu^{A} \star \mu_\varepsilon ^{A}$, where $ \mu_\varepsilon ^{A}$ is the image measure
 of  $\mu_\varepsilon$ by $A$. Moreover, the probability measure estimator $\hat \mu^A_n$ can be written
$\hat \mu ^A_n=  m (Z_1, \dots, Z_n)$ where $Z_i=A Y_i$ and  $m$ is a measurable function from $(\R^d) ^n$  into the set of  probability measures on $\R^d$. Thus,
$$   \E _{ (\mu \star \mu_\varepsilon) ^{\otimes n }} \,  W_1 \left(  \hat \mu ^{A}, \mu^{A}   \right)
 =   \E _{ (\mu^{A} \star \mu_\varepsilon^{A} ) ^{\otimes n }} \,
   W_1 \left( m (Z_1, \dots, Z_n)   ,  \mu^{A}   \right)   .
$$
Since $ \mu \in \mathcal C_{A}  (M,p,a) \Leftrightarrow   \mu^{A} \in \mathcal C (M,p,a)$, we obtain that
\begin{equation}\label{lb1}
\sup_{\mu \in \mathcal C_A (M,p,a)} \,  \E _{ (\mu \star \mu_\varepsilon) ^{\otimes n }} \,
W_1(\hat \mu_n, \mu)
  \geq \|A\|^{-1}  \sup_{\mu^A \in \mathcal C (M,p,a)}\, \E _{ (\mu^{A} \star \mu_\varepsilon^{A} ) ^{\otimes n }} \,     W_1  \left( m (Z_1, \dots,
Z_n)
\, , \, \mu^{A}   \right) \, . \\
\end{equation}
Note that, in the model $Z_i=AX_i + A\varepsilon_i$, the error $\eta=A \varepsilon$
has independent coordinates  and satisfies the assumptions of Theorem \ref{theo:LBWp}
for $A = \operatorname{I}_d$.

We now apply the lower bound obtained in Section \ref{eic}, which gives that there exists a positive constant $C$
such that
\begin{equation}\label{lb2}
\liminf_{ n \rightarrow \infty} \ (\log n) ^{1 / \beta}
 \sup_{\nu \in \mathcal C (M,p,a)}\, \E _{ (\nu \star \mu_\eta ) ^{\otimes n }} \,     W_1 \left( m (Z_1, \dots, Z_n)  \, , \, \nu   \right) \geq C
.
\end{equation}
The result follows from (\ref{lb2}) and (\ref{lb1}).


\section{Upper bounds} \label{sec:upperB}
In this section, we generalize the results of \cite{CaillerieEtAl2011}
by proving an upper bound on the rates of convergence for the estimation of the probability $\mu$ under any metric  $W_p$.

\subsection{Errors with independent coordinates}\label{DE}

In this section, we assume that the random variables $(\varepsilon_{1,j})_{1\leq j \leq d}$ are independent, which means that $\varepsilon_1$ has the distribution
$\mu_\varepsilon= \mu_{\varepsilon,1} \otimes \dots \otimes \mu_{\varepsilon,d}$.

Let $p \in [1, \infty[$ and denote by $\lceil p \rceil$ the smallest integer
greater than $p$.
We first  define a kernel $k$ whose Fourier transform  is smooth enough and compactly supported over $[-1,1]$.
 Such  kernels can be defined by considering powers of the sinc function. More precisely,
let
\begin{equation*}\label{thekernel}
   k (x)=c_p
  \left\{ \frac{(2 \lceil p/2 \rceil +2)\sin \frac x {2 \lceil p/2 \rceil +2} }{x}\right \}^{2 \lceil p/2 \rceil +2} \, .
\end{equation*}
where $c_p $ is such that $\int k(x) d x = 1$.
The kernel $k$ is a symmetric density, and  $k^*$ is supported over $[-1,1]$. Moreover
$k^*$  is $\lceil p \rceil$ times differentiable with Lipschitz $\lceil p \rceil$-th derivative.  For any $j \in \{1, \cdots , d \}$
and any $h_j>0$, let
$$
  \tilde k_{j,h_j}(x)= \frac{1}{2 \pi} \int e^{iux} \frac{k^*(u)}{\mu_j^*(u/h_j)} du \, .
$$
A preliminary estimator $\hat f_n$ is given by
\begin{equation}\label{fnhatbis}
  \hat f_n(x_1, \ldots, x_d)= \frac{1}{n} \sum_{i=1}^n \prod_{j=1 \ldots d} \frac 1 h_j \tilde{k}_{j,h_j}
  \Big(\frac{x_j-Y_{i,j}}{h_j}\Big)\, .
\end{equation}
The estimator (\ref{fnhatbis}) is the multivariate version of the standard
deconvolution kernel density estimator which was first introduced in \cite{CarrollHall88}. This estimator has been the
subject of many works in the one dimensional case, but only few authors have studied the multidimensional deconvolution problem, see
 \cite{Tang94}, \cite{ComteLacour2013} and \cite{CaillerieEtAl2011}.

 The estimator $\hat f_n$ is not necessarily a density, since it has no reason
to be non negative. Since our estimator has to be a probability measure, we
define
 $$
   \hat g_n(x)= \alpha_n \hat f_n^+(x), \quad  \text{where} \quad
   \alpha_n=\frac{1}{\int_{{\mathbb R}^d} \hat f_n^+(x) dx} \quad \text{and} \quad
   \hat f_n^+=\max \{0, \hat f_n\}  \, .
 $$
 The estimator $\hat \mu_n$ of $\mu$ is then the probability measure with density $\hat g_n$.

The next theorem gives the rates of convergence of the estimator $\hat \mu_n$
under some assumptions on the derivatives of the functions
 $r_j :=  1 / {\mu_{\varepsilon,j}^*}$.

\begin{thm} \label{upperbound} Let $M >0$, $p\geq 1$ and $a>1$.
 Assume that we observe a $n$-sample $Y_1\dots,Y_n$ in the multivariate convolution model (\ref{modelConv}) and
that  $\E | \varepsilon_j|^{2p+a} < \infty $ for all $j \in \{1,\dots,d\}$. Also assume that
there
exists  $\beta>0$, $\tilde \beta \geq 0$, $\gamma_2 > 0$ and $c_2 >0$ such
that for every $j  \in \{1,\dots,d\}$, every $\ell  \in \{0,1 \dots,\lceil p \rceil   +1\}$ and every $t\in \R$:
\begin{equation}\label{CdtRj}
\left|  r_j ^{(\ell)}(t) \right| \leq c_2 (1+ |t|^{\tilde \beta}) \exp\left(  |t|^\beta  / \gamma_2\right).
\end{equation}
Taking $h_1 = \dots = h_d= (4d/(\gamma_2 \log (n))^{1/\beta}$, there exists a positive constant $C$
such that
\begin{equation*}\label{BorneSupW1}
\sup_{\mu \in \mathcal D (M,p,a)}\,  \E_{ (\mu \star \mu_\varepsilon) ^{\otimes n }} \left( W_p ^p \left(\mu, \hat \mu_n \right) \right)  \leq C \left(
\log n \right) ^{- \frac
p \beta} \, .
\end{equation*}
\end{thm}

\subsubsection{\bf Proof of Theorem \ref{upperbound}}
Let $\bh =(h_1, h_2, \ldots, h_d)$.
We follow the proof of Proposition 2 in \cite{CaillerieEtAl2011}. First we have the bias-variance
decomposition
$$
  \E_{ (\mu \star \mu_\varepsilon) ^{\otimes n }} (W^p_p(\hat \mu_n, \mu))\leq
  2^{p-1} B(\bh) +
  2^{2(p-1)} \int_{{\mathbb R}^d} (2^{p-1}C(\bh)+\|x\|^p)\sqrt{\mathrm{Var}(\hat f_n(x))}dx \, ,
$$
where
$$
B(\bh)=  \int \|\bh^tx \|^p K(x) dx \quad \text{and}
\quad C(\bh)= B(\bh)+ \int \|x\|^p \mu(dx) \, .
$$
The proof of this inequality is the same as that of Proposition 1 in \cite{CaillerieEtAl2011},
by using Theorem 6.15 in \cite{villani2008oton}.

Note that  $B(\bh)$ is such that
$B(\bh) \leq d^{p-1} \beta ( h_1^p + \dots + h_d^p)$, with $\beta= \int |u|^p k(u) du$.
To ensure the consistency of the estimator, the bias term $B(\bh)$ has to tend to zero as
$n$ tends to infinity. Without loss of generality, we assume in the following that
$\bh$ is such that $B(\bh)\leq 1$.  Hence, the variance term
$$
V_n=2^{2(p-1)} \int_{{\mathbb R}^d} (2^{p-1}C(\bh)+\|x\|^p)\sqrt{\mathrm{Var}(\hat f_n(x))}dx
$$
is such that
$$
V_n \leq C\int_{{\mathbb R}^d} \Big(1+\sum_{j=1}^d |x_j|^{ p }
\Big)\sqrt{\mathrm{Var}(\hat f_n(x_1,\ldots, x_n))} \ dx_1\ldots dx_d
$$
for some positive constant $C$.
Now
$$
\sqrt{\mathrm{Var}(\hat f_n(x_1,\ldots, x_n))} \leq \frac {1} {\sqrt n}
 \sqrt{ \E_{ (\mu \star \mu_\varepsilon) ^{\otimes n }} \left[ \left\{ \prod_{j=1}^d \frac{1}{h_j}\tilde k_{j,h_j}
\Big(\frac{x_j-Y_{1,j}}{h_j}\Big) \right\} ^2 \right]} \, .
$$
Using that $\mu \in \mathcal D(M,p,a)$ and applying Cauchy-Schwarz's inequality $d$-times, we obtain
that
\begin{multline*}
\int_{{\mathbb R}^d} \sqrt{\mathrm{Var}(\hat f_n(x_1,\ldots, x_n))} \ dx_1\ldots dx_d \\
 \leq \frac{D_1}{\sqrt n}\sqrt{
  \E_{ (\mu \star \mu_\varepsilon) ^{\otimes n }}  \Big(\prod_{j=1}^d\int  (1\vee | x _j |^a)
\Big(\frac{1}{h_j}\tilde k_{j,h_j}
\Big(\frac{x_j- | Y_{1,j}}{h_j}\Big)\Big)^2 dx_j}\Big)\\
\leq  \frac{D_2}{\sqrt n}\sqrt{  \E_{ (\mu \star \mu_\varepsilon) ^{\otimes n }}
  \Big(\prod_{j=1}^d   (1 \vee |Y_{1,j}|^a)\Big) \prod_{j=1}^d \int(1 \vee |u_j|^a h_j^a)
\frac{1}{h_j}(\tilde k_{j,h_j}(u_j))^2 d u_j}
\end{multline*}
where  $D_1$ and $D_2$ are positive constants depending on $a$ and $d$.
Now, by independence of $X_1$ and $\varepsilon_1$, and by independence of the coordinates of
$\varepsilon_1$, we find that
$$
\E_{ \mu \star \mu_\varepsilon }
  \Big(\prod_{j=1}^d   (1 \vee |Y_{1,j} |^a)\Big)\leq 2^d  
  \E_{ \mu }
  \Big(\prod_{j=1}^d (1\vee |X_{1, j}|^a) \Big)
  \prod_{j=1}^d (1\vee \E_{\mu_\varepsilon }(|\varepsilon_{1, j}|^a)).
$$
Without loss of generality, assume that $a \in (1,2]$.
Using that $\E_{\mu_\varepsilon }(|\varepsilon_{1, j}|^a)  < \infty $ and that  $\mu$ satisfies (\ref{eq:DAM}), it follows that
\begin{multline*}
\int_{{\mathbb R}^d} \sqrt{\mathrm{Var}(\hat f_n(x))} \ dx
 \leq 
\frac{A_0}{\sqrt n}\sqrt{
  \prod_{j=1}^d  \int \! (1 \vee |u_j|^a h_j^a)
  \frac{1}{h_j}(\tilde k_{j,h_j}(u_j))^2 d u_j } \\
  \leq 
\frac{A_0}{\sqrt n}\sqrt{
  \prod_{j=1}^d  \int \! (1 + |u_j|^2 h_j^2)
  \frac{1}{h_j}(\tilde k_{j,h_j}(u_j))^2 d u_j } .
\end{multline*}
In the same way, using again that  $\mu \in \mathcal D(M,p,a)$ and that  $\E | \varepsilon_\ell|^{2p+a} < \infty $, we obtain that
\begin{multline*}\label{M6}
\int_{{\mathbb R}^d} |x_\ell|^{p }\sqrt{\mathrm{Var}(\hat f_n(x_1,\ldots, x_n))}
\ dx_1\ldots dx_d
\\
\leq 
\frac{A_\ell}
{\sqrt n}
\sqrt{
\int (1\vee|u_\ell|^{2 p  +a} h_\ell^{2 p  +a}) \frac{1}{h_\ell}(\tilde k_{\ell,h_\ell}(u_\ell))^2 d u_\ell
 \prod_{ j\neq \ell}  \int(1 \vee |u_j|^a h_j^a)
\frac{1}{h_j}(\tilde k_{j,h_j}(u_j))^2 d u_j }  \\
\leq
\frac{A_\ell}
{\sqrt n}
\sqrt{
\int (1+|u_\ell|^{2 p  +2} h_\ell^{2 p  +2}) \frac{1}{h_\ell}(\tilde k_{\ell,h_\ell}(u_\ell))^2 d u_\ell
 \prod_{ j\neq \ell}  \int(1 + |u_j|^2 h_j^2)
\frac{1}{h_j}(\tilde k_{j,h_j}(u_j))^2 d u_j } .
\end{multline*}

Starting from these computations, one can prove the following Proposition.

\begin{prop}\label{bound} Let
$(h_1, \ldots, h_d) \in [0,1]^d$. The following upper bound holds
$$
\E_{ (\mu \star \mu_\varepsilon) ^{\otimes n }} (W_p^p(\hat \mu_n, \mu)) \leq (2d)^{p-1}\beta (h_1^p+ \dots + h_d^p)
+ \frac{L}{\sqrt n}\left(\prod_{j=1}^d I_j(h_j) + \sum_{\ell=1}^d J_\ell(h_\ell)
\Big(\prod_{j=1, j\neq \ell}^d I_j(h_j)\Big)\right)
$$
where $L$ is some positive constant  and
\begin{eqnarray*}
I_j(h)  & \leq &   \sqrt{\int_{-1/h}^{1/h} (r_j (u))^2 + (r'_j (u))^2 du}
\, ,\\
J_j(h)  & \leq & \sqrt{\int_{-1/h}^{1/h} (r_j (u))^2 + (r_j^{(\lceil p \rceil+1)} (u))^2 du} \\
& &
+ \sum_{k=1}^{\lceil p \rceil}  h^{\lceil p \rceil+1-k}\sqrt{\int_{-1/h}^{1/h}   (r^{(k)}_j (u))^2 du} \, .
\end{eqnarray*}
\end{prop}

Let us finish the proof of Theorem   \ref{upperbound}
  before proving
Proposition \ref{bound}. Take $h_1= \ldots= h_d=h$. The condition
(\ref{CdtRj}) on the derivatives of
$r_j$ leads to the upper bounds
$$
\E_{ (\mu \star \mu_\varepsilon) ^{\otimes n }}(W_p^p(\hat \mu_n, \mu)) \leq C\Big(h^p+ \frac{1}{\sqrt n h^{d(2\tilde \beta+1)/2}} \exp(d/(\gamma_2
h^{\beta}))\Big) \, .
$$
The choice  $h=(4d/(\gamma_2 \log (n))^{1/\beta}$ gives the desired result.

\medskip

\noindent{\bf Proof of Proposition \ref{bound}.}
It follows the proof of Proposition 2 in \cite{CaillerieEtAl2011}.
By Plancherel's identity,
\begin{eqnarray*}
\int
\frac{1}{h}(\tilde k_{j,h}(u))^2 d u = \frac{1}{2 \pi} \int
\frac{1}{h} \frac{({k^*}(u))^2}{(\mu_{\varepsilon,j}^*(u/h))^2} du & = & \frac{1}{2\pi} \int \frac{({k^*}(hu))^2}{(\mu_{\varepsilon,j}^*(u))^2} du \\
& \leq & \frac{1}{2\pi}  \int_{-1/h}^{1/h}  r_j^2(u) du\, .
\end{eqnarray*}
the last upper bound being true because $k^*$ is supported over $[-1,1]$
and bounded by 1.

Let $C$ be a positive constant, which may vary from line to line. Let $q_{j,h}(u)=r_j(u/h)k^*(u)$. Since $q_{j,h}$ is differentiable with
compactly supported derivative, we have that
$$
-i u 2 \pi \tilde k_{j,h}(u)= (q_{j,h}')^*(u) \, .
$$
Applying Plancherel's identity again,
\begin{eqnarray*}
\int
hu^2 (\tilde k_{j,h}(u))^2 d u &=&  \frac{1}{2 \pi} \int
 h(q'_{j,h}(u))^2 du \\
 &\leq C& \Big( \int_{-1/h}^{1/h} (r'_j (u))^2 du  + h^2  \int_{-1/h}^{1/h}   r_j^2(u) du \Big) \, ,
\end{eqnarray*}
the last inequality being true because $k^*$ and $(k^*)'$ are compactly
supported over $[-1,1]$.
Consequently
$$
\sqrt{\int(1 + u_j^2 h_j^2)
\frac{1}{h_j}(\tilde k_{j,h_j}(u_j))^2 d u_j}
\leq C I_j(h_j) \, .
$$

In the same way
$$
(-i u)^{\lceil p \rceil +1} 2 \pi \tilde k_{j,h}(u)= (q_{j,h}^{(\lceil p \rceil +1)})^*(u)
$$
and
$$
\int
h^{2\lceil p \rceil +1}u^{2\lceil p \rceil +2} (\tilde k_{j,h}(u))^2 d u = \frac{1}{2\pi} \int
 h^{2\lceil p \rceil +1}(q_{j,h}^{(\lceil p \rceil +1)}(u))^2 du .
$$
Now, since  $k^*, (k^*)', \ldots,  (k^*)^{(\lceil p \rceil +1)}$  are compactly supported over $[-1,1]$,
\begin{equation*}
 \int h^{2\lceil p \rceil +1}(q_{j,h}^{(\lceil p \rceil +1)}(u))^{2} du \leq C
 \sum_{k=0}^{\lceil p \rceil +1} h^{2(\lceil p \rceil +1-k)}\int_{-1/h}^{1/h} (r_j^{(k)} (u))^2 du \, .
\end{equation*}
Consequently
\begin{multline*}
\int (1+|u_\ell |^{2 p  +2} h_\ell ^{2 p  +2}) \frac{1}{h_\ell}(\tilde k_{\ell,h_\ell}(u_\ell))^2 d u_\ell
\\
\leq 2\int (1+|u_\ell|^{2 \lceil p \rceil  +2} h_\ell^{2 \lceil p \rceil  +2}) \frac{1}{h_\ell}(\tilde k_{\ell,h_\ell}(u_\ell))^2 d u_\ell
\leq C J_\ell (h_\ell) \, .
\end{multline*}
The results follows. $\square$

\subsection{The general case}\label{general}

Here, as in the introduction, we  shall assume that there exists an invertible matrix
$A$ such that the coordinates of the vector $A \varepsilon _1$ are independent.
Applying $A$ to the random variables $Y_i$ in (\ref{modelConv}),
we obtain the new model
$$
 AY_i= A X_i + A \varepsilon_i \, ,
$$
that is: a convolution model in which each error vector
$\eta_i=A \varepsilon_i$ has independent coordinates.

To estimate the image measure $\mu^A$ of $\mu$ by $A$, we use the
preliminary
estimator (\ref{fnhatbis}), that is
$$
  \hat f_{n, A}(x_1, \ldots, x_d)=
  \frac{1}{n} \sum_{i=1}^n \prod_{j=1 \ldots d} \frac 1 h_j \tilde{k}_{j,h_j}
  \Big(\frac{x_j-(AY_{i})_j}{h_j}\Big)\, ,
$$
and the estimator $\hat \mu_{n, A}$ of $\mu^{A}$ is deduced from
$\hat f_{n, A}$ as in Section \ref{DE}. This estimator $\hat \mu_{n, A}$
has the density $\hat g_{n, A}$ with respect to the Lebesgue measure.

To estimate $\mu$, we define $\hat \mu_n=\hat \mu_{n, A}^{A^{-1}}$ as the image measure
of  $\hat \mu_{n, A}$ by $A^{-1}$. This estimator has the density
$\hat g_n=  |A|\hat g_{n, A}\circ A$ with respect to the Lebesgue measure.
It can be deduced from the preliminary estimator
$
  \hat f_n=|A|\hat f_{n, A}\circ A
$
as in Section \ref{DE}. Now
\begin{eqnarray*}
  W_p^p(\hat \mu_n, \mu) &=&
  \min_{\lambda \in \Pi(\hat \mu_n, \mu)} \int \|x-y\|^p \lambda(dx,dy) \\
  &=& \min_{\pi \in \Pi(\hat \mu_{n, A}, \mu^{A})} \int \|A^{-1}(x-y)\|^p \pi(dx,dy)
 \, .
\end{eqnarray*}
Consequently, if $\|A^{-1}\|=\sup_{\|x\|=1} \|A^{-1}x\|$, we obtain that
\begin{equation}\label{comp}
W_p^p(\hat \mu_n, \mu)
  \leq \|A^{-1}\|^p W_p^p(\hat \mu_{n, A}, \mu^{A}) \, ,
\end{equation}
which is an equality if $A$ is an unitary matrix.
Note also  that $\mu \in \mathcal D_A (M,p,a)$ if and only if $\mu^A \in
\mathcal D (M,p,a)$.

Let $\mu_\eta$ be the distribution of the $\eta_i$'s. Since the coordinates of the $\eta_i$'s are independent, $\mu_\eta$ can be written as
$\mu_\eta= \mu_{\eta,1} \otimes \dots \otimes \mu_{\eta,d}$. As in Section
\ref{DE}, let $r_j :=  1 / {\mu_{\eta,j}^*}$.
Assume that the $r_j$'s satisfy the
condition (\ref{CdtRj}).
It follows from (\ref{comp})
 and Theorem  \ref{upperbound} that,
 taking $h_1 = \dots = h_d= (4d/(\gamma_2 \log (n))^{1/\beta}$, there exists a
 positive constant $C$ such that
\begin{equation*}
\sup_{\mu \in \mathcal D_A (M,p,a)}\,  \E_{ (\mu \star \mu_\varepsilon) ^{\otimes n }} \left( W_p ^p \left(\mu, \hat \mu_n \right) \right)  \leq C \left(
 \log n \right) ^{- \frac
p \beta} \, .
\end{equation*}

\subsection{Examples of rates of convergence}

\paragraph{Gaussian noise.}
Assume that we observe $Y_1, \dots, Y_n$ in the multivariate convolution model
(\ref{modelConv}), where $\varepsilon$ is a centered non degenerate Gaussian random vector.
In that case, there always exists
 an invertible matrix $A$ such that the coordinates of $A \varepsilon_1$
are independent. The distribution of
$(A \varepsilon_1)_j$ is  either a Dirac mass at zero or a  centered Gaussian random variable with positive variance.
Since $\varepsilon$ is non degenerate, there exists at
least one index $j_0$
for which $(A \varepsilon_1)_{j_0}$ is non zero.

Now, the distribution of $(A\varepsilon_1)_{j_0}$ satisfies the assumptions
of
Theorem \ref{theo:LBWp}, for any $p\geq 1$ and $\beta=2$ (Conditions (\ref{Assum:supsm+}) and (\ref{eq:pkappa12}) follow from
Lemma~\ref{lem:Cdtions}).
Moreover, denoting by $\mu_{\eta,j}$ the distribution of $\eta_{1,j} = (A \varepsilon_1)_j$, then the quantity
$r^*_j=1/\mu_{\eta,j}^*$ satisfies (\ref{CdtRj}) for any $p\geq 1$ and $\beta=2$.
Theorem \ref{ThG}  follows then from Theorems \ref{theo:LBWp} and
\ref{upperbound} (more precisely, the estimator
$\hat \mu_n$ of Theorem \ref{ThG}  is constructed as in Section \ref{general}).

\paragraph{Other supersmooth distributions.}

For $\alpha \in ]0,2[$, we denote by $s_\alpha$ the symmetric $\alpha$-stable density, whose Fourier transform $q_\alpha$ is given by
$$s_\alpha^*(x)=q_{\alpha} (x)=  \exp \left(-|x|^ {\alpha}\right)\, .
$$
Let $q_{\alpha,1}=q_\alpha$ and $q_{\alpha,2}= q_\alpha \star q_\alpha$. For any positive integer $k >2$,
define by induction
$q_{\alpha,k}=q_{\alpha,k-1} \star q_\alpha$.

\begin{lem} \label{lem:convolexp}
Let $k$ be a positive integer. 
There exists two positive constants $a_{\alpha, k}$ and
$b_{\alpha,k}$ such that for any $x \in \R$,
$$
a_{\alpha,k} \exp  \left(-|x|^ {\alpha} \right)   \leq q_{\alpha, k}(x)
 \leq b_{\alpha,k}  \exp\left(- \frac{|x|^ {\alpha}}{ 2 ^{(k-1)\alpha}}\right) .
$$
\end{lem}

The proof of Lemma~\ref{lem:convolexp} is given in Appendix \ref{ap:C}. Next, for any integer $k \geq 2$, we introduce the supersmooth density
$$
f_{\alpha,k}(x)= \frac{(s_\alpha (x))^k}{\int (s_\alpha (x))^k dx},
$$
 and we note that
$f_{\alpha, k}^*= q_{\alpha, k}/q_{\alpha,k}(0)$
and
$ f_{\alpha,k} (x)  = O(|x|^{-k(\alpha+1)})$ 
by the well known properties of $\alpha$-stable densities (see for instance Section 1.6 in \cite{CaillerieEtAl2011}).
Note that the density $ f_{\alpha,k}$ has a moment of order 
$m$ for any integer $m<-1+k (\alpha+1)$, so that $f^*_{\alpha,k}$ 
is $m$ times differentiable with bounded derivatives.
Let $r_{\alpha, k}= 1/f^*_{\alpha,k}$. It follows that
$$
   |r_{\alpha, k}^{(\ell)}(x)| \leq C_\ell \sum_{i=1}^{\ell}
   \frac{|(f^*_{\alpha,k})^{(i)}(x)|}{(f^*_{\alpha,k}(x))^{\ell+2-i}} \quad \text{ for $1\leq \ell \leq m$.}
$$
Applying Lemma
\ref{lem:convolexp}, for any $\ell \in \{0, \ldots, m\}$,
\begin{equation}\label{ralpha1}
|r_{\alpha, k}^{(\ell)}(x)| \leq K_{\alpha, \ell} \exp(|x|^\alpha) \, .
\end{equation}
 Moreover, we also have the lower bound
\begin{equation}\label{ralpha2}
       |r_{\alpha, k}(x)| \geq c_{\alpha, k}
       \exp\left( \frac{|x|^ {\alpha}}{ 2 ^{(k-1)\alpha}}\right) \, .
\end{equation}

Now, assume that we observe $Y_1, \dots, Y_n$ in the multivariate convolution model (\ref{modelConv}). Let $p\geq 1$ and $a>1$,
and  assume that there exists an
invertible matrix $A$ such that, for any $j \in \{1, \ldots, d \}$, $(A \varepsilon_1)_j$ has  the distribution
$f_{\alpha_j, k_j}$ for some $\alpha_j \in ]0,2[$ and some positive
integer $k_j$ such that 
$k_j > (2p+a+1)/(\alpha_j +1)$. This choice of $k_j$ implies that
$(A \varepsilon_1)_j$ has a moment of order $2p+a$. 
Let $\alpha= \max_{1 \leq j \leq d} \alpha_j$.

Inequality (\ref{ralpha2}) gives Condition (\ref{Assum:supsm}) in Theorem~\ref{theo:LBWp} for $\beta = \alpha$. Since
$(A \varepsilon_1)_j$ has a moment of order $2p+a$ Lemma~\ref{lem:Cdtions} can be
applied, and  Conditions (\ref{Assum:supsm+}) and (\ref{eq:pkappa12})
of Theorem~\ref{theo:LBWp} are  satisfied. 
Note that  (\ref{ralpha1}) is satisfied for any 
$\ell \in \{0, \ldots, m\}$ and any integer 
$m$ such that $m<2p+a$. Hence
Condition (\ref{CdtRj})  in Theorem~\ref{upperbound} is satisfied
for $\beta = \alpha$. Theorems \ref{theo:LBWp} and \ref{upperbound} finally give the following result:
\begin{enumerate}
\item
There exists a constant $C>0$  such that for all estimator $\tilde \mu_n$ of the
measure $\mu$:
$$
\liminf _{ n \rightarrow \infty} \  (\log n) ^{p / \alpha}   \sup_{\mu \in \mathcal D_A(M,p, a)}\,  \E_{ (\mu \star \mu_\varepsilon) ^{\otimes n }}  (
W_p^p (\tilde \mu_n, \mu) ) \geq
C .
$$
\item The estimator $\hat \mu_n$  of $\mu$ constructed in Section
\ref{general} is such that
$$
\sup_{n \geq 1} \sup_{\mu \in \mathcal D_A(M,p,a)} \ (\log n) ^{p / \alpha} \
\E_{ (\mu \star \mu_\varepsilon) ^{\otimes n }}  ( W_p^p (\hat \mu_n, \mu) ) \leq
K \, ,
$$
for some positive constant $K$.
\end{enumerate}

%
%

\paragraph{Mixtures of distributions.}

Of course, the independent coordinates of $A \varepsilon_1$ need not all be
Gaussian or even supersmooth.

For instance if there exists $j_0$ such that $(A\varepsilon_1)_{j_0}$ is a non degenerate
Gaussian random variable, and the other coordinates have distribution which is either
a Dirac mass at $0$ or a Laplace distribution, or a supersmooth distribution
$f_{\alpha, k}$ for some $\alpha \in ]0,2[$ and
$k > (2p+a+1)/(\alpha +1)$ (this list
in non exhaustive), then the
 estimator
$\hat \mu_n$ of $\mu$ constructed in Section \ref{general} is
such that
$$
\sup_{n \geq 1} \sup_{\mu \in \mathcal D_A(M,p,a)} \ (\log n) ^{p / 2} \
\E_{ (\mu \star \mu_\varepsilon) ^{\otimes n }}  ( W_p^p (\hat \mu_n, \mu) ) \leq
K \, ,
$$
and this rate is minimax.

In the same way if there exists $j_0$ such that $(A\varepsilon_1)_{j_0}$ is
supersmooth with density $f_{\alpha, k}$ for some $\alpha \in ]0,2[$ and
$k > (2p+a+1)/(\alpha +1)$, and the other coordinates have distribution which is either
a Dirac mass at $0$ or a Laplace distribution, or a supersmooth distribution
$f_{\beta, m}$ for some $\beta \in ]0, \alpha]$ and
$ m > (2p+a+1)/(\beta +1)$, then
the estimator
$\hat \mu_n$
of $\mu$ constructed in Section \ref{general} is such that
$$
\sup_{n \geq 1} \sup_{\mu \in \mathcal D_A(M,p,a)} \ (\log n) ^{p / \alpha} \
\E_{ (\mu \star \mu_\varepsilon) ^{\otimes n }}  ( W_p^p (\hat \mu_n, \mu) ) \leq
K \, ,
$$
and this rate is minimax.

\section{Discussion} \label{sec:Disc}

In the supersmooth case, we have seen that lower bounds for
the Wasserstein deconvolution problem in any dimension can be deduced from lower bounds for the deconvolution of the c.d.f in dimension one. But this method cannot work
in the ordinary smooth case for $d>1$, because, contrary to the supersmooth case, the rates of convergence depend on the dimension.

Let us briefly discuss the case where $d=1$ and the error distribution is
ordinary smooth.  It is actually well known that establishing optimal rates
of convergence in the ordinary smooth case is more difficult than in the supersmooth case, even for
pointwise estimation, as noticed by Fan in \cite{Fan91}.
 When the density is $m$ times differentiable, Fan gives in this
paper pointwise lower  and upper bounds for the  estimation of the c.d.f. in both the supersmooth case and the ordinary smooth case. He finds the
optimal rates in the supersmooth case and he conjectures that his upper bound is actually optimal in the ordinary smooth case (see his Remark 3). Optimal
pointwise rates for the deconvolution of the c.d.f. in the ordinary smooth case was an open question until recently. This problem has been solved in
\cite{DattnerEtAl2011} when the density belongs to a Sobolev class.

When $d=1$ and the error distribution is ordinary smooth, some results
about integrated rates of convergence for the density (and its derivatives) can be found in \cite{Fan93,Fan91b} but the case of the c.d.f. (for the
integrated risk) is not studied in these papers. However, some lower bounds can be easily computed by
following the method of \cite{Fan93} and using the pointwise rates of \cite{Fan91} : for  a class of ordinary smooth noise densities of order $\beta$ and assuming only that
the unknown distribution $\mu$ has a moment of order $4$, we
find that the minimax integrated risk is lower bounded by  $n^{-1/(2 \beta +1)}$ and we then obtain the same lower bound for $W_1$.
As for the pointwise estimation described in \cite{Fan91}, these rates do not match with the upper bounds given by  Proposition \ref{bound} for $W_1$. For instance, for Laplace errors ($\beta =
2$), the rate of convergence of the kernel estimator under $W_1$ is upper bounded by $n^{-1/7}$. We are currently working on this issue.

\appendix

\section{Some known lemmas}\label{A}

\medskip
\noindent The following lemma is given in \cite{FanTruong93} (Lemma~1):
\begin{lem} \label{Lemma1FanTruong93} Let $  H $ be a function such that $$ |  H(t) | \leq C ( 1+
t^2)^{-r} $$
for some $C >0 $ and some $r > 0.5$. Then there exists a positive constant $\tilde C$ such that for any sequence $b_n \rightarrow \infty $,
$$\sum _{s=1} ^{b_n}   |  H \left(b_n ( t -  s / n) \right)| \leq \tilde C (1+ t^2) ^{-r}. $$
\end{lem}

\noindent Let $f_{0,r} $ be the function defined in (\ref{eq:f0}). The following lemma can be found in \cite{Fan91} (Lemma~5.1):
\begin{lem} \label{Lemma51Fan91} For any probability measure $\mu$, there exists a constant $C_r >0$ such that
$$f_{0,r} \star \mu (t) \geq C_r t^{-2r} \quad  \textrm{as $|t|$ tends to infinity}.$$
\end{lem}

\noindent The following lemma is rewritten from \cite{Fan91} (Lemma~5.2):
\begin{lem} \label{Lemma52Fan91}
Let $r >0$. Suppose that $P( | \varepsilon_1 ' - t | \leq  |t|^{\kappa_1} ) = O ( |t| ^{-\kappa_2})$ as  $|t|$ tends to infinity for some $0 <
\kappa_1 < 1$ and $\kappa_2> 1$. Let $H$ be a bounded function such that $ | H(t) | \leq O ( |t| ^{-2r})$ for some $r >  \kappa_2 / (2 \kappa_1).$ Then there exists a large  $T$ and a constant $C$ such that when $|v| / b_n \geq T$ :
$$   \int_{-\infty} ^{+\infty}  H (v-y) g(y/b_n)\, d y  / b_n   \leq C (|v|/b_n) ^{-\kappa_2}  . $$
\end{lem}

\section{Distances between probability measures}\label{B}

The first lemma follows straightforwardly from the definition of $W_1$.

\begin{lem} \label{MinorW1}
Let $\mu$ and $\tilde \mu$ be two measures on $\R^d$ with finite first moments, and let ${\mu}_1$ and ${\tilde \mu}_1$ be their first marginals. Then
$
W_1(\mu, \tilde \mu) \geq W_1\left( \mu_1 , \tilde \mu _1 \right)
$.
\end{lem}

\noindent The following Lemma is a particular case of the famous Le Cam's inequalities. See for instance Section 2.4 in~\cite{Tsybakov09} for more details.
\begin{lem} \label{lem:LeCam}
Let  $h$ and $\tilde h$ be two densities  on $\R^n$, then
$$ \int_{\R^n}  \min \left( h(x) , \tilde h(x) \right)  d  x \geq  \frac 1 2  \left\{ \int_{\R^n}  \sqrt{ h(x)  \tilde h(x) }  d  x  \right\}^2  . $$
 \end{lem}
\noindent  The next lemma can  be found for instance in Section~2.4 of \cite{Tsybakov09}.

\begin{lem}  \label{lem:lbchi2}
Let $h$ and $\tilde h$ be two densities for the Lebesgue measure on $\R$, then
$$  \int _\R \sqrt{ h(y) \, \tilde h(y) } \, dy  \geq  1 - \frac 1 2 \chi ^2 (  h  , \tilde h) .$$
\end{lem}

\section{Auxiliary results} \label{ap:C}

\paragraph{Proof of Lemma~\ref{lem:convolexp}.}
It suffices to prove Lemma~\ref{lem:convolexp} for $k=2$, and the general case follows by induction.
Since $q_\alpha \star  q_\alpha$ is symmetric, it suffices to prove the result for $x>0$.
Now, for any $x >0 $,
\begin{eqnarray*}
 q_{\alpha} \star  q_{\alpha}(x) &=& 2 \int_{x/ 2} ^{+\infty} \exp \left(-|x-t|^ {\alpha}-
 t^ {\alpha} \right) d t \\
&\leq& a_{\alpha, 2} \exp \left(-(x/2)^ {\alpha}  \right).
\end{eqnarray*}
On the other hand,  for any $x >1$, there exist a positive constant $c_\alpha$ such that
\begin{eqnarray}
 q_{\alpha} \star  q_{\alpha}(x) &\geq &   \int_{x/ 2} ^{x} \exp \left(-|x-t|^ {\alpha} -t^ {\alpha} \right) d t \notag \\
&\geq & \exp \left(-x^ {\alpha}  \right) \int_{0} ^{x/ 2} \exp \left( - u^ {\alpha} \right) d u  \notag \\
&\geq& c_{ \alpha} \exp \left(-x^{\alpha}  \right)  \label{ref:CA1} .
\end{eqnarray}
The function $x \mapsto  q_{\alpha} \star  q_{\alpha}(x) \exp \left(x^ {\alpha}
\right)$ is continuous and positive on $[0,1]$ and thus (\ref{ref:CA1}) is also true on  $[0,1]$ for some other positive  constant $c'_{\alpha}$.
 The lower bound follows by taking
$b_{\alpha,2}= \min \{ c_{ \alpha}, c'_{\alpha} \}$.


\bibliographystyle{alpha}
\bibliography{BornesInfDeconvW}

\end{document}